\documentclass[a4paper,12pt]{amsart}

\usepackage[left=3cm,right=3cm, bottom=3cm]{geometry} 
\usepackage{url}
\usepackage{textcomp}
\usepackage{amsmath}
\usepackage{amsthm}
\usepackage{amssymb}
\usepackage{stmaryrd}
\usepackage[french, british]{babel}
\usepackage[utf8]{inputenc} 
\usepackage[T1]{fontenc}
\usepackage{tikz}
\usepackage{centernot}
\usepackage{mathrsfs}
\usepackage[all]{xy}
\usepackage{calligra}
\usepackage{nicefrac}
\usepackage{mathtools}
\usepackage{hyperref}
\hypersetup{
    colorlinks=true,
    linkcolor=blue,
    filecolor=magenta,      
    urlcolor=cyan,
}

\theoremstyle{plain}
\newtheorem{theorem}{Theorem}[section]

\newtheorem{lemma}[theorem]{Lemma}
\newtheorem{corollary}[theorem]{Corollary}
\newtheorem{proposition}[theorem]{Proposition}
\newtheorem{conjecture}[theorem]{Conjecture}

\theoremstyle{definition}
\newtheorem{remark}[theorem]{Remark}

\theoremstyle{definition}
\newtheorem{definition}[theorem]{Definition}
\newtheorem{example}[theorem]{Example}

\DeclareMathOperator{\Aut}{Aut}

\DeclareMathOperator{\Hom}{Hom}

\DeclareMathOperator{\Id}{Id}

\DeclareMathOperator{\Pic}{Pic}

\DeclareMathOperator{\Ext}{Ext}

\DeclareMathOperator{\ch}{ch}
\DeclareMathOperator{\td}{td}
\DeclareMathOperator{\chern}{c}

\DeclareMathOperator{\D}{D}

\def\dar[#1]{\ar@<2pt>[#1]\ar@<-2pt>[#1]}

\DeclareMathOperator{\Coh}{\mathbf{Coh}}

\newcommand\matR{\mathbb{R}}
\newcommand\matQ{\mathbb{Q}}

\newcommand\matZ{\mathbb{Z}}
\newcommand\matC{\mathbb{C}}

\newcommand\matP{\mathbb{P}}

\title[Categorical vs topological entropy on surfaces]{Categorical vs topological entropy of autoequivalences of surfaces}
\author{DOMINIQUE MATTEI}
\address{Institut de Math\'ematiques de Toulouse ; UMR5219 \newline 
\indent UPS, F-31062 Toulouse Cedex 9, France}
\email{dmattei@math.univ-toulouse.fr}

\begin{document}
\selectlanguage{british}
\maketitle





\begin{abstract}
In this paper, we give an example of an autoequivalence with positive categorical entropy (in the sense of Dimitrov, Haiden, Katzarkov and Kontsevich) for any surface containing a $(-2)$-curve. Then we show that this equivalence gives another counter-example to a conjecture proposed by Kikuta and Takahashi. In a second part, we study the action on cohomology induced by spherical twists composed with standard autoequivalences on a surface $S$ and show that their spectral radii correspond to the topological entropy of the corresponding automorphisms of $S$.
\end{abstract}
\vspace{1\baselineskip}

\section{Introduction}

%

Recently, Dimitrov-Haiden-Katzarkov-Kontsevich \cite{MR3289326} introduced a categorical analogue of the topological entropy, namely the \textit{categorical entropy} (see section \ref{catentropy}), in the context of triangulated categories. A typical example of such is the derived category $\D^b(X):=\D^b(\Coh(X))$ of coherent sheaves on a variety $X$. In \cite{MR3903564}, Kikuta and Takahashi proposed a Gromov-Yomdin (\cite{MR2026895}, \cite{MR880035}, \cite{MR889979}) type conjecture:

\begin{conjecture}\label{KTConj}
Let $X$ be a smooth projective variety over $\matC$. For any autoequivalence $\phi \in \Aut(\D^b(X))$, we have
$$h_0(\phi) = \log \rho(HH_\bullet(\phi))$$
where $h_0$ is the categorical entropy (valued in $0$), $HH_\bullet(\phi)$ is the $\matC$-linear isomorphism induced by $\phi$ on the Hochschild homology group $HH_\bullet(X)$ and $\rho$ denotes the spectral radius.
\end{conjecture}

The lower bound $\log \rho(HH_\bullet(\phi)) \le h_0(\phi)$ is always true (see \cite{KST2018}). 

Conjecture \ref{KTConj} is known to be true when $X$ is a curve \cite{MR3600071}, an abelian surface \cite{2017arXiv170104009Y}, a variety with ample (anti)-canonical bundle \cite{MR3903564}. However, counterexamples have been found, first by Fan \cite{MR3826832} for strict Calabi-Yau manifold of dimension $\ge 3$, and then by Ouchi \cite{2017arXiv170501001O} for any $K3$ surface.

We will replace $HH_\bullet(\phi)$ by the induced Fourier-Mukai cohomological morphism $\phi^H : H^*(X,\matC) \to H^*(X,\matC)$ thanks to  the modified Hochschild-Konstant-Rosenberg isomorphism (\cite{MR2534995}, Theorem $1.2$).

%

In the case of surface we show the following, which is our first main result.

\begin{theorem}\label{maintheorem}
Let $S$ be a smooth projective surface and $C\subseteq S$ a $(-2)$-curve. Let $\mathcal{L}\in\Pic(S)$ be a line bundle satisfying $\deg_C(\mathcal{L}|_C)<0$ and consider the autoequivalence $\varphi = T_{\mathcal{O}_C}\circ (-\otimes \mathcal{L})$. Then we have
$$h_0(\varphi) > 0 = \log \rho(\varphi^H).$$
\end{theorem}

Note that, if the surface $S$ contains a $(-2)$-curve $C$, then $\mathcal{L}:=\mathcal{O}_S(C)$ fits the hypothesis. As a consequence, this gives a counterexample of Conjecture \ref{KTConj} in any birational class of surfaces. 

The inequality $h_0(\varphi) >0$ is proved in section \ref{minorationentropy} and the equality $\log \rho(\varphi^H) = 0$ is proved at the end of section \ref{actioncoh} (see Corollary \ref{corbaseautoequiv}).
\vspace{1\baselineskip}




Let $S$ be a smooth projective complex surface. In the second part of this paper, we inspect more generally the action on cohomology induced by autoequivalences of $\D^b(S)$.

In \cite{MR1689873}, Cantat shows that if $S$ admits an automorphism $f\in\Aut(S)$ of positive topological entropy, then $S$ is birational to either $(i)$ $\matP^2$, $(ii)$ a $K3$ surface, $(iii)$ a $2$-dimensional complex torus or $(iv)$ an Enriques surface. In the case $(i)$, $S$ is a blow up of $\matP^2$ at $10$ or more points.


We aim to find an analogue of this theorem relying the birational nature of the surface $S$ with the action on cohomology of its group of autoequivalences $\Aut(\D^b(S))$.

Recall that \textit{standard equivalences} are the equivalences lying in the subgroup
\begin{eqnarray*}
\Pic(S)\rtimes \Aut(S) \times \matZ\cdot [1] \subseteq \Aut(\D^b(S)).
\end{eqnarray*}
The classical result of Bondal and Orlov states that equality holds when $\pm K_S$ is ample. 

Define $B:=\langle T_{\mathcal{O}_C(a)} | \ C (-2)\text{-curve}, \ a\in\matZ \rangle$.
Our second main result is the following.

\begin{theorem}[= Theorem \ref{theoremactioncoh}]\label{theoremactioncohintro}
Let $S$ be a smooth surface for which its union $Z$ of $(-2)$-curves is a disjoint union of finite configurations of type $A$-$D$-$E$.
Let $\varphi\in  \langle B, \Pic(S) \rangle \rtimes \Aut(S)   \ \times \matZ\cdot [1]$ be an autoequivalence, so that, up to a shift, we have a decomposition
$$\varphi = b \circ (-\otimes\mathcal{L}) \circ f^*$$
with $b\in B, \mathcal{L}\in\Pic(S), f\in\Aut(S)$.

Then $$\rho(\varphi^H)=\rho(f^*).$$
\end{theorem}

In \cite{MR3896121}, Uehara proposes a trichotomy for surfaces in order to understand their groups of autoequivalences. In the case of a surface $S$ for which $K_S \not\equiv 0$ and $S$ admits no minimal elliptic fibration ($N_S=2$ in his notation), he conjectured the following description.

\begin{conjecture}[\cite{MR3896121}, Conjecture $1.2$]
Denote by $B_Z \subseteq \Aut(\D^b(S))$ the subgroup generated by spherical twists along objects supported on the union $Z$ of all $(-2)$-curves. Then
$$\Aut(\D^b(S)) = \langle B_Z, \Pic(S) \rangle \rtimes \Aut(S) \times \matZ[1].$$
\end{conjecture}
The conjecture is proved (\cite{MR3896121}, Theorem $6.8$) when $Z$ is a disjoint union of $A$-configurations. In this case, $B_Z=B$ (\cite{MR2198807}, Corollary $6.10$).
Using this, we obtain a first Cantat-type corollary of Theorem \ref{theoremactioncohintro}.

\begin{corollary}
Let $S$ be a smooth surface with finitely many $(-2)$-curve in disjoint $A$-configurations. Assume $K_S\not\equiv 0$ and that $S$ admits no minimal elliptic fibration. Then, if there is an autoequivalence $\varphi\in \Aut(\D^b(S))$ with $\rho(\varphi^H)>1$, $S$ is rational. 
\end{corollary}
\vspace{1\baselineskip}

\subsection*{Notations}

For a smooth projective variety $X$, we denote $\D^b(X)$ the bounded derived category of coherent sheaves of $X$. 

All the functors we consider are derived but (often) written with their classical notation (e.g. $\otimes$ instead of $\otimes^L$). Recall that closed immersions have no higher direct image, and that the tensor product and the pull back need not to be derived when applied to locally free sheaves.

For any complex $F\in\D^b(X)$, $j\in\matZ$, we write $\mathcal{H}^j(F)$ the $j$-th cohomology sheaf of the complex. For $F,G\in\D^b(X)$, we denote $\Ext^j_X(F,G) := \Hom_{\D^b(X)}(F,G[i])$ the higher Hom groups.
\newline

%

\section{Preliminaries}

\subsection{Categorical entropy}\label{catentropy}

Let $K$ be a field and $\mathcal{T}$ be a $K$-linear triangulated category of finite type. The definitions and results exposed here come from \cite{MR3289326}, but are also explained in Haiden's lectures notes \cite{HaidenLectNote}, section $9$.

Let $A,B\in\mathcal{T}$ be non-zero objects. If $B\in \langle A \rangle$, where $\langle - \rangle$ denotes the split closure, we can construct a tower of triangles
\begin{equation*}
\xymatrix{
 0 \ar[rr] & & B_1 \ar[dl] \ar[rr] & & B_2 \ar[dl] & \cdots &  B_{k-1} \ar[rr] & & B\oplus B^{'} \ar[dl]\\
 & A[n_1] \ar@{-->}[ul] &  & A[n_2] \ar@{-->}[ul] & & \cdots & & A[n_k] \ar@{-->}[ul]  &
}
\end{equation*}
for some $B' \in \mathcal{T}$, with $k\ge 0$ and $n_i\in\matZ$.

\begin{definition}
We define the \textit{complexity} of $A$ relative to $B$ as the following function: for all $t\in\matR$,
$$\delta_t(A,B):= \inf \left\{ \sum_{j=1}^k e^{n_jt}\right\}\in \matR \cup \{+\infty\}$$
where the infimum is taken over all possible towers.

Note that $\delta_t(A,B)=+\infty$ for all $t$ if and only if $B \notin \langle A \rangle$.
\end{definition}


\begin{proposition}[\cite{MR3289326}, Prop. $2.2$]
For any non-trivial $A,B,C\in \mathcal{T}$ we have the following:
\begin{itemize}
\item $\delta_t(A,B)$ depends on $A$ and $B$ only up to isomorphisms,
\item $\delta_t(A,C)\le \delta_t(A,B)\delta_t(B,C)$,
\item If $\mathcal{T}^{'}$ is a triangulated category of finite type and $F : \mathcal{T}\to \mathcal{T}^{'} $ is an exact functor, then $\delta_t(FA,FB)\le \delta_t(A,B)$.
\end{itemize}
\end{proposition}

\begin{definition}
Let $G$ be a split generator of $\mathcal{T}$ and $\phi : \mathcal{T} \to \mathcal{T}$ an exact endofunctor such that $\phi^n\neq0$ for all $n\ge0$. The \textit{categorical entropy} of $\phi$ is defined to be the function
$$h_t(\phi) := \lim_{n \to \infty} \frac{1}{n} \log \delta_t(G,\phi^n G).$$
\end{definition}

The following lemma is very useful.

\begin{lemma}[\cite{MR3289326}, Lemma $2.5$]\label{lemmaentropy}
The limit \ $\lim_{n \to \infty} \frac{1}{n} \log \delta_t(G,\phi^n G)$ exists in $[-\infty,+\infty)$ for every $t\in\matR$ and is independant of the choice of the split-generator $G$.
Moreover, if $G^{'}$ is another split-generator, then
$$h_t(\phi) = \lim_{n \to \infty} \frac{1}{n} \log \delta_t(G,\phi^n G^{'}).$$
\end{lemma}

In the case of the derived category of a smooth projective variety $X$, the entropy can be computed as Poincar\'e polynomials in $\Ext$ groups.

\begin{proposition}[\cite{MR3289326}, Thm. $2.6$]
For any autoequivalence $\phi : \D^b(X) \to \D^b(X)$ and for any split-generators $G,G^{'}$ we have
$$h_t(\phi) = \lim_{n \to \infty} \frac{1}{n} \log \left(\sum_{j\in\matZ} \dim \Ext^j(G,\phi^nG')e^{jt}\right).$$
\end{proposition}
\vspace{0.5\baselineskip}


\subsection{Spherical twists}

Let $X$ be a smooth projective variety over $\matC$. Spherical twists, introduced by Seidel and Thomas \cite{MR1831820}, are a really important example of autoequivalence of $\D^b(X)$.

\begin{definition}
An object $\mathcal{E}\in\D^b(X)$ is called \textit{spherical} if $\mathcal{E}\otimes\omega_X\simeq\mathcal{E}$ and $R\Hom(\mathcal{E},\mathcal{E}) \simeq \matC \oplus \matC[-\dim X]$.
\end{definition}

\begin{example}
Let $X=S$ be a surface, and $C \xhookrightarrow{i} S$ a $(-2)$-curve, that is, a curve $C\simeq\matP^1$ with self-intersection $-2$. Then any line bundle $\mathcal{O}_C(a)$, $a\in\matZ$, is a spherical object. Indeed, $K_S\cdot C=0$ by adjunction formula and thus $i_*(\mathcal{O}_C(a))\otimes\omega_S \simeq i_*(\mathcal{O}_C(a)\otimes i^*\omega_S) \simeq i_*\mathcal{O}_C(a)$ by projection formula. Now, using \cite{huybrechts2006fourier}, section $11$:
\begin{eqnarray*}
R\Hom(i_*\mathcal{O}_C(a),i_*\mathcal{O}_C(a)) &=& 
R\Hom(i^*i_*\mathcal{O}_C(a),\mathcal{O}_C(a))\\
&=&
R\Hom(\mathcal{O}_C(a)\oplus \mathcal{O}_C(a+2)[1],\mathcal{O}_C(a))\\
&=& 
R\Hom(\mathcal{O}_C,\mathcal{O}_C\oplus \mathcal{O}_C(-2)[1]),
\end{eqnarray*}

and the result comes from direct computation of these $\Ext$-groups since $C\simeq \matP^1$.
\end{example}

\begin{definition}
The \textit{spherical twist} $T_\mathcal{E}$ with respect to a spherical object $\mathcal{E}\in\D^b(X)$ is the integral functor with kernel 
$$\text{Cone}(\eta : q^*\mathcal{E}^\vee \otimes p^*\mathcal{E} \to \mathcal{O}_{\Delta_X})$$
 where $p,q$ are the natural projections $X\times X \to X$ and $\eta$ is the natural pairing. On the level of objects it is given by
$$A \mapsto T_\mathcal{E}(A) := \text{Cone}(R\Hom(\mathcal{E},A)\otimes\mathcal{E} \xrightarrow{\text{ev}} A),$$
i.e. $T_\mathcal{E}(A)$ is given by the cone of the natural evaluation map.
\end{definition}

\begin{proposition}\cite{MR1831820}
The spherical twist $T_\mathcal{E}$ is an equivalence.
\end{proposition}
\vspace{1\baselineskip}

\section{Minoration of the categorical entropy}\label{minorationentropy}

Let $S$ be a smooth complex projective surface, $C \xhookrightarrow{i} S$ a $(-2)$-curve.
Let $\mathcal{L}\in\Pic(S)$ be a line bundle verifying $\deg_C(\mathcal{L}|_C)=l<0$. For instance, $\mathcal{L}=\mathcal{O}_S(C)$ satisfies this assumptions.
Consider the autoequivalence
$$\varphi=T_{\mathcal{O}_C}\circ \mathcal{L}.$$
The goal of this section is to show the following:

\begin{theorem}\label{positiveentropy}
The categorical entropy of $\varphi$ verifies
$$h_0(\varphi)>0.$$
\end{theorem}

First, we make some constructions.
For any $\mathcal{M}\in\Pic(S)$ we have the distinguished triangle
$$R\Hom(i_*\mathcal{O}_C,\varphi^{n-1}(\mathcal{M})\otimes\mathcal{L}) \otimes i_*\mathcal{O}_C \to \varphi^{n-1}(\mathcal{M})\otimes\mathcal{L} \to \varphi^n(\mathcal{M}).$$

Now pick $\mathcal{P}\in\Pic(S)$ and apply $(-\otimes \mathcal{P})$ and $R\Hom(i_*\mathcal{O}_C, -)$ to this triangle. We obtain:
\begin{align}\label{triangle}
R\Hom(i_*\mathcal{O}_C,&\varphi^{n-1}(\mathcal{M}) \otimes \mathcal{L})\otimes R\Hom(i_*\mathcal{O}_C,i_*\mathcal{O}_C\otimes\mathcal{P}) \\
&\to R\Hom(i_*\mathcal{O}_C,\varphi^{n-1}(\mathcal{M})\otimes\mathcal{L}\otimes\mathcal{P}) \nonumber \\
 &\to R\Hom(i_*\mathcal{O}_C,\varphi^n(\mathcal{M})\otimes\mathcal{P}). \nonumber
\end{align}

Fix $\deg_C(\mathcal{M}|_C)=m<0$. We consider the triangle (\ref{triangle}) depending on the parameter $p:=\deg_C(\mathcal{P}|_C)$.
For more clarity, we introduce the following notations.
\begin{align*}
&\widetilde{A}_n:=R\Hom(i_*\mathcal{O}_C,\varphi^{n-1}(\mathcal{M})\otimes \mathcal{L}), \\
&D(p):=R\Hom(i_*\mathcal{O}_C,i_*\mathcal{O}_C\otimes\mathcal{P}),\\
&A_n(p):=\widetilde{A}_n \otimes D(p),\\
&B_n(p) := R\Hom(i_*\mathcal{O}_C,\varphi^{n-1}(\mathcal{M}) \otimes \mathcal{L}\otimes \mathcal{P}),\\
&C_n(p) := R\Hom(i_*\mathcal{O}_C,\varphi^n(\mathcal{M})\otimes\mathcal{P}).
\end{align*}

Thus the triangle (\ref{triangle}) can be written as:
\begin{equation}
\label{triangle2}
A_n(p) \to B_n(p) \to C_n(p).
\end{equation}

\begin{proposition}\label{recprop}
For all $n\ge 1$ and any $\mathcal{P}$ with $p<0$, we have
$$\mathcal{H}^{j}(C_n(p))=0 \text{ for } j>n+2$$
and moreover
$$\mathcal{H}^{n+2}(C_n(p))\simeq \mathcal{H}^{n+1}(C_{n-1}(l))\otimes \mathcal{H}^2(D(p)) \neq 0.$$
\end{proposition}

\begin{proof}

Let's start with computations for $n=1$.

By adjunction formula, $\deg_C(i^*\omega_S)=0$ since $K_S\cdot C = 0$. Now we use the adjunction 
$i_* \dashv i^*(-) \otimes\omega_C [-1]$
(see \cite{huybrechts2006fourier}, Proposition 3.35) and compute $i^*i_*\mathcal{O}_C$ using \cite{huybrechts2006fourier}, section $11$ again. We obtain:

\begin{eqnarray*}
\Ext^k_S(i_*\mathcal{O}_C,\mathcal{M}\otimes\mathcal{L}\otimes\mathcal{P}) &=& \Ext^k_C(\mathcal{O}_C,\mathcal{O}_C(m+l+p-2)[-1]) \\
&=& H^{k-1}(C,\mathcal{O}_C(m+l+p-2)).\nonumber
\end{eqnarray*}
\begin{eqnarray*}
\Ext^k_S(i_*\mathcal{O}_C,i_*\mathcal{O}_C\otimes\mathcal{P}) &=& \Ext^k_C(\mathcal{O}_C,i^*i_*(\mathcal{O}_C)\otimes\mathcal{O}_C(p-2)[-1])\\
&=& \Ext^k_C(\mathcal{O}_C,(\mathcal{O}_C\oplus\mathcal{O}_C(2)[1])\otimes\mathcal{O}_C(p-2)[-1])\nonumber \\
&=& H^{k-1}(C,\mathcal{O}_C(p-2)) \oplus H^{k}(C,\mathcal{O}_C(p)).\nonumber
\end{eqnarray*}
\vspace{0.01mm}

Since we fixed $m<0, l<0, p<0$, these Ext groups are non-zero only for $k=2$ (and possibly $k=1$ if $p<-1$).
Hence we have:
\begin{itemize}
\item $\mathcal{H}^j(\widetilde{A}_1)\neq 0$ only for $j=2$, 
\item $\mathcal{H}^j(D(p))\neq 0$ only for $j=2$ (and $j=1$ if $p<-1$) and thus $\mathcal{H}^j(A_1(p))\neq 0$ only for $j=4$ (and $j=3$ if $p<-1$),
\item $\mathcal{H}^j(B_1(p))\neq 0$ only for $j=2$.
\end{itemize}

Using the long exact sequence in cohomology induced by (\ref{triangle2}) we have
$$\mathcal{H}^j(C_1(p))\neq 0 \text{ only for } j=2,3 \text{\ and \ } \mathcal{H}^3(C_1(p)) \simeq \mathcal{H}^4(A_1(p)),$$

as it can be read on the following table:

\begin{center}
\begin{tabular}{ l c c r }
    & $\mathcal{H}^2$ & $\mathcal{H}^3$ &  $\mathcal{H}^4$ \\
   $A_1$ & 0 &      & $\ast$ \\
   $B_1$ & $\ast$ & 0 & 0\\
   $C_1$ & $\ast$ & $\ast$ & 0\\
 \end{tabular}
\end{center}
where $\ast$ means that the space does not vanish, and the empty slots are irrelevant to our calculations.

For any $n\ge 1$ we have the identities
\begin{eqnarray}\label{identities}
\widetilde{A}_{n} \simeq C_{n-1}(l) \text{ and } B_{n}(p)\simeq C_{n-1}(l+p).
\end{eqnarray}

For $n=1$, by (\ref{identities}) we get
\begin{eqnarray*}
\mathcal{H}^4(A_1(p)) &\simeq& \mathcal{H}^2(\widetilde{A}_1) \otimes \mathcal{H}^2(D(p))\\
&\simeq& \mathcal{H}^2(C_0(l)) \otimes \mathcal{H}^2(D(p))\\
&\neq& 0.
\end{eqnarray*}

Assume that the lemma is true for all $p<0$ on rank $n-1$.
Since $l$ and $p$ are negative, by induction hypothesis and (\ref{identities}) we have
\begin{itemize}
\item $\mathcal{H}^j(A_n(p))=0$ for $j>n+3$, 
\item $\mathcal{H}^j(B_n(p))=0$ for $j>n+1$,
\item $\mathcal{H}^{n+3}(A_n(p)) \simeq \mathcal{H}^{n+1}(\widetilde{A}_n) \otimes \mathcal{H}^2(D(p)) \neq 0$. 
\end{itemize}
Thus using the long exact sequence in cohomology induced by (\ref{triangle2}) we obtain
$$\mathcal{H}^{n+2}(C_n(p))\simeq \mathcal{H}^{n+3}(A_n(p)).$$
Once again this can be read on the table:
\begin{center}
\begin{tabular}{ l c c r }
    & $\mathcal{H}^{n+1}$ & $\mathcal{H}^{n+2}$ & $\mathcal{H}^{n+3}$ \\
   $A_n$ &  &   & $\ast$ \\
   $B_n$ & $\ast$ & 0 & 0\\
   $C_n$ & $ $  & $\ast$ & 0\\
 \end{tabular}
\end{center}

Finally by the identities (\ref{identities}), we obtain
$$\mathcal{H}^{n+2}(C_n(p)) \simeq \mathcal{H}^{n+1}(C_{n-1}(l)) \otimes \mathcal{H}^2(D(p)).$$

\end{proof}

\begin{corollary}\label{corexpoext}
For any $\mathcal{M},\mathcal{P}$ with $m,p<0$ and $n\in\matZ_{\ge 1}$, we have
\vspace{2mm}
\begin{eqnarray*}
\Ext^{n+2}_S(i_*\mathcal{O}_C,\mathcal{P} \otimes \varphi^n(\mathcal{M})) \simeq H^1(C,\mathcal{O}_C(m+l-2)) \otimes H^1(C,\mathcal{O}_C(p-2)) \\
\otimes H^1(C,\mathcal{O}_C(l-2))^{\otimes n-1}.
\end{eqnarray*}

\vspace{1.5mm}
In particular, $\dim \Ext^{n+2}_S(i_*\mathcal{O}_C,\mathcal{P} \otimes \varphi^n(\mathcal{M})) > (1-l)^{n-1}$.
\end{corollary}

\begin{proof}
By induction on Proposition \ref{recprop}, we have
\begin{eqnarray*}
\Ext^{n+2}_S(i_*\mathcal{O}_C,\mathcal{P} \otimes \varphi^n(\mathcal{M})) &\simeq& \mathcal{H}^{n+2}(C_n(p)),\\
&\simeq& \mathcal{H}^{n+1}(C_{n-1}(l)) \otimes \mathcal{H}^2(D(p),\\
&\simeq& \mathcal{H}^{2}(C_0(l)) \otimes \mathcal{H}^2(D(p)) \otimes \mathcal{H}^2(D(l))^{\otimes n-1},\\
&\simeq& H^1(C,\mathcal{O}_C(m+l-2)) \otimes H^1(C,\mathcal{O}_C(l-2))^{\otimes n-1}\\
&& \otimes H^1(C,\mathcal{O}_C(p-2)).
\end{eqnarray*}
\end{proof}

\begin{proof}[Proof of theorem \ref{positiveentropy}]

Recall that by Orlov \cite{MR2567400}, for any very (anti)-ample line bundle $\mathcal{M}$ on a smooth projective variety $X$, the vector bundle $\mathcal{M}\oplus \mathcal{M}^{\otimes 2}\oplus \dots \oplus \mathcal{M}^{\otimes \dim X+1}$ is a generator of $\D^b(X)$.

In our case, we fix a generator $G=\mathcal{M} \oplus \mathcal{M}^{\otimes 2} \oplus \mathcal{M}^{\otimes 3}$ of $\D^b(S)$ with $\mathcal{M}\in\Pic(S)$ so that $\mathcal{M}^\vee$ is very ample on $S$. Thus $m:=\deg_C(\mathcal{M}|_C)<0$. Now choose a line bundle $\mathcal{P}\in\Pic(S)$ so that $\mathcal{P}^\vee$ is very ample. Up to taking powers of $\mathcal{P}^\vee$, we can assume that $\mathcal{P}_0^\vee := \mathcal{P}^\vee \otimes \mathcal{O}_S(-C)$ is also very ample (see \cite{hartshorne2013algebraic}, II, ex. $7.5$). Then $G_1:= \mathcal{P}^\vee \oplus (\mathcal{P}^\vee)^{\otimes 2} \oplus (\mathcal{P}^\vee)^{\otimes 3}$ and $G_2 :=  \mathcal{P}_0^\vee \oplus (\mathcal{P}_0^\vee)^{\otimes 2} \oplus (\mathcal{P}_0^\vee)^{\otimes 3}$ are also generator of $\D^b(S)$.

By Corollary \ref{corexpoext}, for any $n\ge 1$ we have 
$$(1-l)^{n-1} \le \dim \Ext^{n+2}(i_*\mathcal{O}_C,\mathcal{P}\otimes \varphi^n(\mathcal{M}) \le \dim \Ext^{n+2}(i_*\mathcal{O}_C,\mathcal{P}\otimes \varphi^n(G)).$$
\vspace{2mm}
Write $\delta_0'(F,G):= \sum_{j\in\matZ} \dim\Ext^j(F,G)$. Considering the exact sequence
$$0 \to \mathcal{O}_S(-C) \to \mathcal{O}_S \to i_*\mathcal{O}_C \to 0,$$
we obtain
\vspace{2mm}
\begin{eqnarray*}
(1-l)^{n-1} \le \delta_0^{'}(i_*\mathcal{O}_C,\mathcal{P} \otimes \varphi^n(G)) &\le& \delta_0^{'}(\mathcal{O}_S, \mathcal{P} \otimes \varphi^n(G)) + \delta_0^{'}(\mathcal{O}_S(-C),\mathcal{P} \otimes \varphi^n(G)),\\
&\le& \delta_0^{'}(\mathcal{P}^\vee, \varphi^n(G)) + \delta^{'}_0(\mathcal{P}^\vee_0, \varphi^n(G)),\\
&\le& \delta_0^{'}(G_1, \varphi^n(G)) + \delta^{'}_0(G_2, \varphi^n(G)),
\end{eqnarray*}
\vspace{0.1mm}

Thus, either $\delta_0^{'}(G_1, \varphi^n(G))$ or $\delta^{'}_0(G_2, \varphi^n(G))$ has exponential growth. By Lemma \ref{lemmaentropy} both terms can be used to compute the categorical entropy $h_0(\varphi)$, hence $$h_0(\varphi)>0.$$

\end{proof}

\begin{remark}
The same result is also true with $T_{\mathcal{O}_C(a)} \circ \mathcal{L}$, for $a$ a non-zero integer: one may perform the same proof with the care of choosing a line bundle $\mathcal{P}$ verifying $\deg_C(\mathcal{P}|_C)=p \ll 0$.
\end{remark}

\begin{remark}
It is interesting to remark that the functor $(-\otimes \mathcal{L})$ can be realized as composition of spherical twists $T_{\mathcal{O}_C(a_1)}\circ \dots \circ T_{\mathcal{O}_C(a_n)}$ with a nice choice of $a_1,\dots,a_n \in \matZ$. See (\cite{MR2198807}, Lemma $4.15$) for the claim. In particular, composition of spherical twists might have positive entropy.
\end{remark}

To finish the proof of Theorem \ref{maintheorem}, it remains to show that the action of $\varphi$ on the cohomology of $S$ has spectral radius $1$. This will be treated in section \ref{actioncoh} (see Corollary \ref{corbaseautoequiv}).

\section{Action of spherical twists on cohomology}\label{actioncoh}

\subsection{Action on cohomology}
We briefly recall how to describe the action on the cohomology of a variety induced by an autoequivalence of its derived category of coherent sheaves.
The definitions and proofs of claims can be found in \cite{huybrechts2006fourier}, section $5.3$. 

Let $X$ be a smooth projective variety over $\matC$. For an object $F \in \D^b(X)$, we define the \textit{Mukai vector} of $F$ as the cohomology class $v(F):=\ch(F)\sqrt{\td F}\in H^*(X,\matQ)$.

\begin{definition}
For a Fourier-Mukai transform $\varphi\in\Aut(\D^b(X))$ with kernel $\mathcal{P}\in\D^b(X\times X)$, we define its \textit{action on cohomology} $\varphi^H$ as the cohomological Fourier-Mukai morphism
$$\varphi^H : H^*(X,\matQ) \to H^*(X,\matQ)$$
with kernel $v(\mathcal{P})$.
\end{definition}

This definition is functorial: $\varphi^H$ is an automorphism of $H^*(X,\matQ)$, and for another autoequivalence $\psi\in\Aut(\D^b(X))$, we have $(\psi \circ \varphi)^H=\psi^H \circ \varphi^H$.

For $v=\sum_j v_{2j}\in \bigoplus_j H^{2j}(X,\matQ)$, we denote by $v^\vee := \sum_j (-1)^j v_{2j}$. We define the \textit{Mukai pairing} on $H^{2*}(X,\matQ)$ as the quadratic form
$$\langle v,w \rangle := \int_X \exp(\chern_1(X)/2)\cdot v^\vee \cdot w,$$
where the integral symbol means taking the top degree part via the identification $H^{2\dim X}(X,\matQ) \simeq \matQ$.
\newline

\subsection{The case of spherical twists and standard equivalences}
Let $S$ be a smooth projective complex surface with finitely many $(-2)$-curves in disjoint union of $A$-$D$-$E$ configurations.

We define $B := \langle T_{\mathcal{O}_C(a)} | \ C \ (-2)\text{-curve}, \ a\in\matZ \rangle$ a subgroup of the group of autoequivalences of $\D^b(S)$, and
$$G := \langle B, \Pic(S)  \rtimes \Aut(S) \rangle \times \matZ[1].$$

\begin{theorem}[=Theorem \ref{theoremactioncohintro}]\label{theoremactioncoh}
Let $\varphi\in G$ be an autoequivalence. 
Then, up to a shift, we have a (not necessarily unique) decomposition
$$\varphi = b \circ (-\otimes\mathcal{L}) \circ f^*$$
with $b\in B, \mathcal{L}\in\Pic(S), f\in\Aut(S)$, and 
$$\rho(\varphi^H)=\rho(f^*).$$
\end{theorem}

\begin{proof}

Pick $\varphi\in G$. First, we show that $\varphi$ admits a decomposition as stated. For the next lemma, see (\cite{huybrechts2006fourier}, Lemma 8.21).

\begin{lemma}\label{lemmatwistcommute}
For any smooth projective variety $X$, for any spherical object $\mathcal{E}\in\D^b(X)$ and for any autoequivalence $\phi \in \Aut(\D^b(X))$, we have
$$\phi \circ T_\mathcal{E} \simeq T_{\phi(\mathcal{E})} \circ \phi.$$
\end{lemma}

Hence, when $\phi$ is an autoequivalence belonging to $\Pic(S)\rtimes \Aut(S)$ and $C$ is a $(-2)$-curve, we shall consider, for $a\in\matZ$, the image $\phi(\mathcal{O}_C(a))$.

First, for any $\mathcal{L}\in\Pic(S)$, we have $\mathcal{L}\otimes\mathcal{O}_{C}(a)=\mathcal{O}_{C}(a+l)$ with $l:=\deg_{C}(\mathcal{L}|_{C})$.

Secondly, consider an isomorphism $f : S \to S$. It induces an isomorphism $\bar{f} : C \to C'$ for some $(-2)$-curve $C'$ as the image of a $(-2)$-curve must be a $(-2)$-curve, and it's easy to check, writing $i$ and $j$ the natural inclusion of $C$ and $C'$ respectively, that $f^*(j_*\mathcal{O}_{C'}(a)) \simeq i_*(\bar{f}^*(\mathcal{O}_{C'}(a))) \simeq i_*(\mathcal{O}_{C}(a))$.


We conclude from this that $B$ is normal in $G$. In fact, we also have $B\cap \Aut(S) = \{0\}$ (\cite{MR2198807}, Remark $4.17$). Hence, up to a shift, $\varphi$ decomposes as
$$\varphi=b\circ (-\otimes \mathcal{L}) \circ f^*,$$
with $b\in B, \mathcal{L}\in\Pic(S)$, $f\in\Aut(S)$ and such $f$ does not depend on the decomposition.
\newline

Note that $\rho((\varphi^H)^{\circ m})=\rho(\varphi^H)^m$, i.e. the spectral radius of the morphism is totally determined by the spectral radius of its powers. Hence, we can assume that $f$ preserves each $(-2)$-curve: $f$ acts by permutation on the set of $(-2)$-curves which is finite, thus some power of $f$ fixes each of them.

We set
$$b=T_{\mathcal{O}_{C_1}(a_1)} \circ \cdots \circ T_{\mathcal{O}_{C_k}(a_k)}$$
with $a_1,\dots,a_k\in\matZ$ and $C_1,\dots,C_k$ $(-2)$-curves. 

Now, we use Lemma \ref{lemmatwistcommute}. Since $B$ is normal in $G$, for any $m\ge 1$ there is a line bundle $\mathcal{L}_m\in \Pic(S)$ and equivalences $b_j\in B$, $j=2,\dots,m$ such that
$$\varphi^{\circ m} = b\circ b_2\circ \cdots \circ b_m \circ (-\otimes \mathcal{L}_m) \circ (f^*)^{\circ m},$$
where each $b_j$ is given by
$$b_j = T_{\mathcal{O}_{C_1}(a_1^{'})}\circ \cdots \circ T_{\mathcal{O}_{C_k}(a_k^{'})}$$
for some integers $a_1^{'},\dots,a_k^{'} \in\matZ$ (depending on $j$). In other words, each $b_j$, $j=2,\dots,m$ is a composition of spherical twists along line bundles over the same curves but with different degrees.

We introduce the following notation: for a morphism $g : H^*(S,\matQ) \to H^*(S,\matQ)$, we denote by $g_2 : H^2(S,\matQ) \to H^2(S,\matQ)$ the restriction to $H^2(S,\matQ)$ of its composition with the projection $p : H^*(S,\matQ) \to H^2(S,\matQ)$.

\begin{proposition}\label{reduc2}
We have
$$\rho((\varphi^H)^{\circ m}) = \max\left( \rho ((\varphi^H)^{\circ m}_2), \rho(f^*)^m\right ).$$
\end{proposition}

\begin{proof}
We write $1\in H^0(S,\matQ)$ the natural generator, $[x]\in H^4(S,\matQ)$ its dual. We fix a graded basis $(1,e_1,\dots,e_k, [x])$ of $H^{2*}(X,\matQ)$ composed by homogeneous elements.

We have $(f^*)^H = f^*$ (pullback in cohomology) and $(-\otimes \mathcal{L}_m)^H =(- \cdot \exp\chern_1(\mathcal{L}_m))$ (\cite{huybrechts2006fourier}, ex. $5.37$). Thus the matrix of $(-\otimes\mathcal{L}_m)^H\circ (f^{*\circ m})^H$ is lower-triangular by blocks, each block corresponding to a graded component of $H^*(S,\matQ)$.

Fix an integer $a\in\matZ$. We shall compute $T_{\mathcal{O}_C(a)}^H$. By Grothendieck-Riemann-Roch, we have $v(\mathcal{O}_C(a)) = [C]+(a+1)[x]$, where $[C]$ denotes the cohomology class of the cycle $C$.

Now, by \cite{huybrechts2006fourier}, $T_{\mathcal{O}_C(a)}$ acts as the identity on $H^{odd}(S,\matQ)$, and for any $w\in H^{2*}(S,\matQ)$ we have

\begin{eqnarray}\label{twistoncoh}
T_{\mathcal{O}_C(a)}^H(w) = w - \langle v(\mathcal{O}_C(a)), w \rangle v(\mathcal{O}_C(a)).
\end{eqnarray}

Denote $w=(w_0,w_2,w_4)\in \matQ \oplus H^2(S,\matQ) \oplus \matQ$. As $\chern_1(S)\cdot [C] = -K_S\cdot C = 0$, we get

\begin{eqnarray}\label{mukpair}
\langle v(\mathcal{O}_C(a)), w\rangle &=& \int_S (0,-[C], a+1)\cdot (w_0,w_2,w_4)\cdot \exp(c_1(S)/2) \nonumber \\
&=& \int_S \bigg(0,-w_0[C], -[C]\cdot w_2 + (a+1)w_0\bigg) \cdot \left(1, \frac{c_1(S)}{2} , \frac{c_1(S)^2}{8}\right) \nonumber \\
&=& (a+1)w_0 -[C]\cdot w_2.
\end{eqnarray}
\vspace{0.01\baselineskip}

This scalar only depends on $w_0$ and $w_2$, and $v(\mathcal{O}_C(a))$ has components only in degree $2$ and $4$ so by (\ref{twistoncoh}) we conclude that $T_{\mathcal{O}_C(a)}^H$ acts as identity on $H^{4}(S,\matQ)$, and by (\ref{twistoncoh}) and (\ref{mukpair}) we see that $T_{\mathcal{O}_C(a)}^H(1)=1+R$ with $R\in H^{\ge2}(S,\matQ)$.

Hence the matrix of $(\varphi^{\circ m})^H$ is triangular by blocks, and both spherical twists and tensors by line bundles have spectral radius $1$ on $H^{j}(S,\matQ)$, $j\neq 2$. We obtain the result.
\end{proof}

\begin{lemma}\label{indepon2}
For any $(-2)$-curve $C$ and any $a\in\matZ$, the map
$$(T^H_{\mathcal{O}_C(a)})_2 : H^2(S,\matQ) \to H^2(S,\matQ)$$
does not depend on $a$. 
\end{lemma}

\begin{proof}
By (\ref{mukpair}), we see that for all $w_2\in H^2(S,\matQ)$,  we have $(T^H_{\mathcal{O}_C(a)})_2(w_2) = w_2+([C]\cdot w_2)[C]$.
\end{proof}

Hence, all the morphisms $b^H, b_j^H$, $j=2,\dots,m$, restrict to the same morphism $(b^H)_2 : H^2(S,\matQ) \to H^2(S,\matQ)$, so we obtain

$$(\varphi^H)_2^{\circ m} = (b^H)_2^{\circ m} \circ (-\otimes \mathcal{L}_m)^H_2 \circ (f^{*})^{\circ m}_2.$$
\vspace*{0.01mm}

\begin{proposition}[\cite{MR1831820}]\label{relationcox}
For any two distinct $(-2)$-curves $C_1$ and $C_2$ in $S$, we have
\begin{align*}
T_{\mathcal{O}_{C_1}} \circ T_{\mathcal{O}_{C_{2}}} \circ T_{\mathcal{O}_{C_1}} &\simeq T_{\mathcal{O}_{C_{2}}} \circ T_{\mathcal{O}_{C_1}} \circ T_{\mathcal{O}_{C_{2}}} &\text{ if } C_1\cdot C_2 = 1, \\
T_{\mathcal{O}_{C_1}} \circ T_{\mathcal{O}_{C_2}} &\simeq T_{\mathcal{O}_{C_2}} \circ T_{\mathcal{O}_{C_1}} &\text{ if } C_1\cdot C_2 = 0.
\end{align*}
\end{proposition}

\begin{proof}
The proof is a consequence of Lemma \ref{lemmatwistcommute}. For more details, see (\cite{huybrechts2006fourier}, Prop. $8.22$).
\end{proof}

%
%

We combine the relations in Proposition \ref{relationcox} and the fact that $(T_{\mathcal{O}_C}^H)^{\circ 2}=\Id$ by (\ref{twistoncoh}) to conclude that the group $\langle (T_{\mathcal{O}_C}^H)_2 \ | \ C \text{ a } (-2)\text{-curve} \rangle \subseteq \Aut(H^2(S,\matQ))$ is a quotient of a finite direct product of Coxeter groups of type $A$,$D$ and $E$. In particular, it is finite, thus $(b^H)_2^{\circ m} = \Id$ for some $m\gg 0$. We obtain
$$(\varphi^H)_2^{\circ m} = (-\otimes \mathcal{L}_m)^H_2 \circ (f^{*})^{\circ m}_2.$$
To conclude the proof, note that we can choose a basis for which both $(-\otimes \mathcal{L}_m)^H_2$ and $(f^{*})^{\circ m}_2$ are lower-triangular, with $(-\otimes \mathcal{L}_m)^H_2$ having only $1'$s on the diagonal.

\end{proof}

As a corollary of Theorem \ref{theoremactioncoh}, we see that the autoequivalence studied in section \ref{minorationentropy} acts in cohomology with spectral radius $1$.

\begin{corollary}\label{corbaseautoequiv}
Let $S$ be any smooth projective surface, $C \hookrightarrow S$ a $(-2)$-curve. Let $\mathcal{L}\in\Pic(S)$ be a line bundle verifying $\deg_C(\mathcal{L}|_C)<0$. Then
$$\rho( T_{\mathcal{O}_C}^H \circ \mathcal{L}^H) = 1.$$
\end{corollary}


\end{document}